\documentclass[12pt]{amsart}
\usepackage{amsmath, amstext, amsbsy, amssymb}
\usepackage{latexsym,array,delarray,epsfig, color,amsfonts,youngtab,yfonts,inslrmin, textcomp, txfonts, pxfonts}
\usepackage{graphicx, epsfig, textcomp, txfonts, pxfonts}
\usepackage{amsmath,amsthm,amsfonts,amssymb,amsxtra, mathtools} 
\usepackage{comment, cancel}
\usepackage{tikz}

\newtheorem{theorem}{Theorem}[section]
\newtheorem{lemma}[theorem]{Lemma}

\newtheorem{proposition}{Proposition}[section]

\theoremstyle{definition}

\theoremstyle{remark}

\numberwithin{equation}{section}

\newcommand{\al}{\alpha}
\newcommand{\be}{\beta}
\newcommand{\si}{\sigma}

\newcommand{\vep}{\varepsilon}

\newcommand{\Z}{\mathbb Z}
\newcommand{\C}{\mathbb C}
\newcommand{\g}{\mathfrak g}
\newcommand{\h}{\mathfrak h}
\newcommand{\cA}{\mathcal A}
\newcommand{\cK}{\mathcal K}
\newcommand{\mf}{\mathfrak}

\def\Za{\mathbb Z+1/2}
\def\Zp{\mathbb Z_++1/2}

\def\dsum{\displaystyle\sum}

\def\d{\delta}
\def\a{\alpha}

\def\ov{\overline}
\def\om{\omega}

\begin{document}

\title[Fermionic realization of twisted toroidal Lie algebras]
{Fermionic realization of twisted toroidal Lie algebras}
\author{Naihuan Jing$^*$} 
\address{Department of Mathematics,
   North Carolina State University,
   Raleigh, NC 27695-8205, USA}
\email{jing@ncsu.edu}
\author{Chad R. Mangum}
\address{School of Mathematical and Statistical Sciences,
   Clemson University, Clemson, SC 29634, USA}
\email{crmangu@clemson.edu}
\author{Kailash C. Misra}
\address{Department of Mathematics,
   North Carolina State University,
   Raleigh, NC 27695-8205, USA}
\email{misra@ncsu.edu}
\thanks{$^*$Corresponding author: Naihuan Jing}
\thanks{Jing acknowledges the support of Simons Foundation grant no. 523868 and National Natural Science Foundation
of China grant no. 11531004, Mangum acknowledges the support of a summer grant at Niagara University,
and Misra acknowledges the support of Simons Foundation grants no. 307555 and 636482}


\begin{abstract}
In this paper, we construct a fermionic realization of the
twisted toroidal Lie algebra of type $A_{2n-1}, D_{n+1}, A_{2n}$ and $D_4$ based on the newly found
Moody-Rao-Yokonuma-like presentation.
\end{abstract}

\maketitle

\section{Introduction} \label{intro}

Affine Kac-Moody Lie algebras \cite{K} are important algebraic structures widely used in mathematics and theoretical physics.
Their introduction was partly due to Kac's study of finite order automorphisms of
simple Lie algebras, and these automorphisms have also led to the notion of twisted affine Lie algebras.
Several important generalizations of affine Lie algebras have been proposed, among them some of the well-known ones are
loop algebras of Kac-Moody Lie algebras \cite{F}, extended affine Lie algebras \cite{ABGP}, toroidal Lie algebras \cite{MRY}
and their multi-loop generalizations \cite{BK, S}. For some recent developments
see the survey \cite{NSW}.

The untwisted toroidal Lie algebras $T({\mathfrak g})$ are certain distinguished algebras
contained in the universal central extension \cite{BK} of the 2-loop algebras of the simple Lie algebra $\mathfrak g$ \cite{MRY}. They
are proved to be nontrivial via the vertex and fermionic representations \cite{FM, T1, T2} as well as
higher level free field realization \cite{JMT}. Other types of $T({\mathfrak g})$-modules have also been constructed in
various works \cite{BB, G, L, FJW}.

Fu and Jiang \cite{FJ} introduced twisted $n$-toroidal Lie algebras in the abstract setting
and studied their integrable modules. In \cite{EM} vertex representations of
general toroidal Lie algebras and Virasoro-toroidal Lie algebras have been considered.
In a recent work \cite{JMM}, we have given an MRY-like presentation for the twisted toroidal Lie algebra
associated with a diagram automorphism of the simple Lie algebra $\mathfrak g = A_{2n-1}, D_{n+1}, D_4$ and have shown that the MRY-like realization is indeed
the universal central extension of the corresponding twisted (baby) toroidal Lie
algebra.

In this paper, we revisit the MRY-like presentation in \cite{JMM} and give the MRY-like presentation for twisted toroidal Lie algebras of type $\mathfrak g = A_{2n-1}, D_{n+1}, A_{2n}, D_4$ in a uniform way. Then we use this new presentation to give a fermionic realization of
these twisted toroidal algebras. Explicitly we use
certain Clifford algebras to realize the twisted toroidal Lie algebras of types $A_{2n+1}, D_{n+1},  A_{2n}, D_4$
twisted by diagram automorphisms of order $2,2,2,3$ respectively. This construction is analogous to the fermionic construction given in the
untwisted situation \cite{JM}, which
in turn was an extension of the fermionic construction of Feingold-Frenkel \cite{FF} for affine Lie algebras; see also \cite{F1, KP} (note that $E_6$ also possesses a finite order diagram automorphism, but, being an exceptional Lie algebra, it was not treated in such works as \cite{FF, F1, JM}, and so is not considered herein).
The level one modules have the degree zero central element $c_1= \overline{t^{-1}dt}$ act as $0$, thus they can be viewed as interesting summation and lifting of level zero modules for the vertical affine Lie subalgebra $\mathfrak g \otimes_{\mathbb C}\mathbb C[t, t^{-1}]\oplus \mathbb Cc_1$. This phenomenon bears some similarity with the famous path construction of level one Fock modules for the affine Lie algebra \cite{DJKMO}, which was an intriguing summation of level zero modules and has led to further work on crystal bases (cf. \cite{JMi}).

The paper is organized as follows. In Section 2, we recall the Fu-Jiang twisted toroidal Lie algebra in a specialized setting.
In Section 3 we give MRY-presentation of twisted toroidal Lie algebras of types $A_{2n+1}, D_{n+1},  A_{2n}, D_4$ with detailed analysis of the
universal central extension for the toroidal Lie algebra. The fermionic free-field constructions of these twisted toroidal Lie algebras are given and proved in Section 4.

The authors would like to thank the referee for useful comments and suggestions which have improved the paper.

\section{Twisted Toroidal Lie Algebras}

Let $\g$ be the finite dimensional simple Lie algebra $A_{2n-1} , (n \geq 3)$,  $D_{n+1} ,  (n \geq 2)$, $A_{2n} , (n \geq 2)$ or $D_4$ over the field of complex numbers $\C$. We denote the Chevalley generators of $\g$ by $\{e_i', f_i', h_i' \mid 1 \leq i \leq N\}$ where $N = 2n - 1, n+1, 2n, 4,$ respectively.  Then $\h' = \text{span}\{h_i' \mid 1 \leq i \leq N\}$ is the Cartan
subalgebra of $\g$. Let $\{\al_i' \mid 1 \leq i \leq N\} \subset \h'^*$  denote the simple roots, $\Delta$ be the set of roots for $\g$, and $Q$ be the root lattice. Note that
$\al_j'(h_i') = a_{ij}'$ where $A' = (a_{ij}')_{i,j = 1}^N$ is the Cartan matrix associated with $\g$. Let $( \ | \ )$ be the nondegenerate symmetric invariant bilinear form on $\g$ defined by $( x | y ) = tr(xy), \frac{1}{2}tr(xy),  tr(xy), \frac{1}{2}tr(xy)$ for all $x, y \in \g$. Then
$( h_i' | h_i' ) = 2 , 1 \leq i \leq N$. Since the Lie algebra $\g$ is simply-laced, we can identify the invariant form on $\h'$ to that on
the dual space $\h'^*$ and normalize the inner product by $( \al | \al ) = 2, \al \in \Delta$.

Let $\Gamma$ denote the Dynkin diagram for $\g$ and $\si$ be the following map on indices of order $r = 2, 2, 2, 3$ respectively:
\begin{eqnarray*}
&&\sigma(i)=N-i+1, i=1, \cdots , N, \ \ \mbox{for type } \ A_{2n-1}  \ or \ A_{2n} \nonumber \\
&&\sigma(i)=i, i=1, \cdots , n-1=N-2; \sigma(n)=n+1,
\ \ \mbox{for type} \ D_{n+1} \nonumber\\
&&\sigma(1,2,3,4)=(3,2,4,1)\ \ \mbox{for type} \
D_4.\nonumber
\end{eqnarray*}


Then $\si$ induces an automorphism of $\Gamma$ via $\si(h_i') = h_{\si(i)}'$ and the Lie algebra $\g$ is decomposed as a ${\Z}/r{\Z}$-graded
Lie algebra:
$${\g}={\g}_0\oplus \cdots\oplus{\g}_{r-1},$$
where ${\g}_i=\{ x\in {\g}| \sigma(x)=\omega^i x\}$ and $
\omega=e^{2\pi\sqrt{-1}/r}$. It is well-known that
the subalgebra ${\g}_0$ is the simple Lie algebra of types
$C_n$, $B_n$, $B_n$ and $G_2$ respectively. Let $I = \{1, 2, \cdots , n\}$ for $\g = A_{2n-1}, D_{n+1}, A_{2n}, D_4$ where $n=2$ for $\g = D_4$.
The Chevalley generators $\{e_i, f_i, h_i \mid i \in I\}$ of $\g_0$ are given by:
\begin{eqnarray*}
&&e_i=e_i', f_i=f_i', h_i=h_i', \ \mbox{if } \sigma(i)=i; \\
&&e_i=\sum_{j=0}^{r-1}e'_{\sigma^j(i)}, \
f_i=\sum_{j=0}^{r-1}f'_{\sigma^j(i)}, \
h_i=\sum_{j=0}^{r-1}h'_{\sigma^j(i)}, \ \mbox{if } \sigma(i)\neq i, \text{ except } i=n \text{ for } A_{2n}. \\
&&e_n=\sqrt{2}(e'_n+e'_{n+1}),
f_n=\sqrt{2}(f'_n+f'_{n+1}),
h_n=2(h'_n+h'_{n+1}) \ \mbox{for } A_{2n}.
\end{eqnarray*}

The Cartan subalgebra of $\g_0$ is $\h_0 = \text{span}\{h_i \mid i \in I\}$ and the simple roots $\{\al_i \mid i \in I\}\subset \h_0^*$ are given by:
$$
\al_i = \frac{1}{r}\sum_{s=0}^{r-1}\al'_{\si^s(i)}.
$$

Then we have
\begin{equation}
(\alpha_i|\alpha_j)=d_ia_{ij},
\ \ \mbox{for all} \ i,j\in I
\end{equation}
where $A = (a_{ij})_{i,j \in I}$ is the Cartan matrix for $\g_0$ and $(d_1, \cdots, d_n)=(1/2, \cdots, 1/2, 1)$, $(1, \cdots, 1, 1/2)$, $(1/2, \cdots, 1/2, 1/4)$
or $(1/3, 1)$,
for $\g = A_{2n-1}, D_{n+1}, A_{2n}$
or  $D_4$ respectively.
Note that $A = (a_{ij})_{i,j \in I}$ is given as follows:
\begin{equation*}
\begin{cases}
a_{ii} = 2, \\
a_{12} = -3, a_{21} = -1, \g_0 = G_2\\
a_{n-1,n} = -2, a_{n,n-1} = -1, \g_0 = C_n\\
a_{n-1,n} = -1, a_{n,n-1} = -2, \g_0 = B_n\\
a_{i,i+1} = a_{i+1,i} = -1, \g_0 \neq G_2 \text{ and } i \neq n-1\\
a_{ij} = 0, \text{otherwise}.
\end{cases}
\end{equation*}

Denote
\begin{eqnarray*}
\theta^0 =\left\{\begin{array}{ll}
\a'_1+\cdots+\a'_{2n-2}  \ \mbox{for } A_{2n-1}, \\
\a'_1+\a'_2+\cdots+\a'_n  \ \mbox{for } D_{n+1}, \\
\a'_1+\cdots+\a'_{2n} \ \mbox{for } A_{2n}, \\
\a'_1+\a'_2+\a'_3 \ \mbox{for } D_4.
\end{array}\right.
\end{eqnarray*}
Let $e'_{\theta^0}, f'_{\theta^0}, h'_{\theta^0}$ denote the $\mathfrak{sl}_2$-triplet associated to $\theta^0$ with bracket $[h'_{\theta^0}, e'_{\theta^0}] = 2 e'_{\theta^0}, [h'_{\theta^0}, f'_{\theta^0}] = -2 f'_{\theta^0}$ and $h'_{\theta^0} = [e'_{\theta^0} , f'_{\theta^0}]$.

Let $\cA = \C[s, s^{-1}, t, t^{-1}]$ be the ring of Laurent polynomials in the commuting variables $s, t$ and $L(\g) = \g \otimes_{\C} \cA$  be the multi-loop algebra with the Lie bracket given by:

$$
[x \otimes s^jt^m , y \otimes s^kt^l] = [x, y] \otimes s^{j+k}t^{m+l},
$$
for all $x, y \in \g, j, k, m, l \in \Z$. For $j \in \Z$ we define $0 \leq \bar{j} < r$ such that $j \equiv \bar{j} \ \mbox{mod} \ r$. For all $j \in \Z$ we define $\g_j = \g_{\bar{j}}$. We extend the automorphism $\si$ of $\g$ to an automorphism $\bar{\si}$ of $L(\g)$ by defining:
$$
\bar\si(x \otimes s^jt^m) = \omega^{-m} \si(x) \otimes s^jt^m
$$
where $x \in \g, j, m \in \Z$. We denote the $\bar\si$ fixed points of $L(\g)$ by $L(\g, \si)$.
Note that the subalgebra $L(\g, \si)$ has the $\Z$-gradation:
$$
L(\g, \si) = \oplus_{m \in \Z} L(\g, \si)_m,
$$
where $L(\g, \si)_m = \g_m \otimes \cA_m, \cA_m = \mbox{span}_{\C}\{s^jt^m \mid j \in \Z\}=t^m\C[s, s^{-1}]$.

Set $F = \cA \otimes \cA$. Then $F$ is a two sided $\cA$-module via the action $a(b_1 \otimes b_2) = ab_1 \otimes b_2 = (b_1 \otimes b_2)a$ for all $a, b_1, b_2 \in \cA$. Let $G$ be the $\cA$-submodule of $F$ generated by $\{1 \otimes ab -a\otimes b -b \otimes a \mid a, b \in \cA\}$. The $\cA$- quotient module $\Omega_{\cA} = F/G$ is called the $\cA$ - module of K{\"a}hler differentials.  The canonical quotient map $d : \cA \longrightarrow \Omega_{\cA}$ given by $da = (1 \otimes a) + G, a \in \cA$ is the differential map. Let $- : \Omega_{\cA} \longrightarrow
\Omega_{\cA}/d\cA = \cK'$ be the canonical linear map. Since $\overline{d(ab)} = 0$, we have $\overline{a(db)} = - \overline{(da)b} = - \overline{b(da)}$ for all  $a, b \in \cA$. Then $\cK' = span_{\C}\{\overline{bda} \mid a, b \in \cA\}$. Set $\cK = span_{\C}\{\overline{bda} \mid a \in \cA_k, b \in \cA_l , k+l \equiv 0 (mod r)\}$ which is a subalgebra of $\cK'$. We note that $\{\overline{s^{j-1}t^mds}, \overline{s^jt^{-1}dt}, \overline{s^{-1}ds}\mid j \in \Z, m \in \Z_{\neq 0}\}$ is a basis for $\cK$ and the following relations are easy to check.
\begin{equation}\label{kahlercalc}
\overline{s^{\ell}ds^k} = \delta_{k, -\ell} k \overline{s^{-1}ds}, \ \
 \overline{s^{\ell}t^{-1}d(s^{k}t)} = \delta_{k, -\ell} k \overline{s^{-1}ds} + \overline{s^{k + \ell}t^{-1}dt}.
\end{equation}
The elements $c_0 = \overline{s^{-1}ds}, c_1= \overline{t^{-1}dt} \in \cK$ are called the degree zero central elements.
Let
$$
T(\g) = L(\g , \si) \oplus \cK,
$$
with the Lie bracket given by
\begin{equation*}
[x\otimes a, y\otimes b] = [x, y]\otimes ab + (x | y)\overline{bda}, \qquad [T(\g), \cK] = 0,
\end{equation*}
where $x \in \g_i, y \in \g_j, a \in \cA_i, b \in \cA_j$ for $i, j \in \Z$. Using \cite[Proposition 2.2]{BK}, it is shown \cite[Theorem 2.1]{FJ} that $T(\g)$, with the canonical projection map $\eta : T(\g) \rightarrow L(\g, \si)$, is the universal central extension of $L(\g, \si)$. $T(\g)$ is called the twisted toroidal Lie algebra of type $\g$. A representation of $T(\g)$ is said to be of level $(k_0, k_1)$ if $c_0$ acts as $k_0(id)$ and $c_1$ acts as
$k_1(id)$.

\section{MRY presentation of $T(\g)$}

In \cite{MRY}, Moody, Rao and Yokonuma gave a presentation of \emph{untwisted} toroidal Lie algebras which is analogous to the Drinfeld realization \cite{D} for quantum affine algebras. In this section we give an MRY type presentation for the \emph{twisted} toroidal Lie algebra $T(\g)$.

Denote $\tilde{I} = I \cup \{0\}$ and extend the Cartan matrix $A = (a_{ij})_{i,j \in I}$ to $\tilde{A} = (a_{ij})_{i,j \in \tilde{I}}$ by defining $a_{00} = 2$, $a_{02} = -1 = a_{20} \ \text{for} \ \g = A_{2n-1}$, $a_{01} = -2, a_{10} = -1 \ \text{for} \ \g = D_{n+1}$, $a_{01} = -1, a_{10} = -2 \ \text{for} \ \g = A_{2n}$, $a_{01} = -1 = a_{10} \ \text{for} \ \g = D_4$ and $a_{0j} = 0 = a_{j0}$ otherwise for all types. Note that $\tilde{A}$ is the Cartan matrix for the twisted affine algebra $\hat{\g}$ of type $A_{2n-1}^{(2)}, D_{n+1}^{(2)}, A_{2n}^{(2)}, D_4^{(3)}$, respectively.
Let $\{ \alpha_i | i \in \tilde{I} \}$, $\hat{Q}$, $\delta$ and $\hat{\Delta}$ denote the simple roots, root lattice, null root and set of roots, respectively for the twisted affine algebra $\hat{\g}$.

Let $t(\g)$ be the Lie algebra over $\C$ generated by symbols
$$\not{c},
 \al_m(k) \text{ and } X(\pm \al_m, k),$$
with $m \in \tilde{I}$ and $k \in \Z$, and satisfying the following relations:
\begin{itemize}
  \item $[\al_0(k), \al_0(l)] =
    \begin{cases}
      2r k \delta_{k,-\ell} \not{c} & (A_{2n-1}, D_{n+1}, D_4) \\
      2k \delta_{k,-\ell} \not{c} & (A_{2n})
    \end{cases}$ \\
  \item $[\al_0(k), \al_j(l)] =
    \begin{cases}
      r a_{0j} k \delta_{k,-\ell} \not{c} & (A_{2n-1}, A_{2n}, D_4) \\
      a_{0j} k \delta_{k,-\ell} \not{c} & (D_{n+1})
    \end{cases}$ \\
where $j \in I$.
  \item $[\al_i(k), \al_j(l)] =
    \begin{cases}
      r a_{ij} k \delta_{k,-\ell} \not{c} & (A_{2n-1}, A_{2n}, D_4) \\
      a_{ij} k \delta_{k,-\ell} \not{c} & (D_{n+1})
    \end{cases}$ \\
where $i,j \in I$ with $i \leq j$ and $(i, j) \neq (n-1, n), (n, n)$.
 \item $[\al_{n-1}(k), \al_n(l)] =
    \begin{cases}
      a_{n-1,n} k \delta_{k,-\ell} \not{c} & (A_{2n-1}, D_4) \\
      4 a_{n-1,n} k \delta_{k,-\ell} \not{c} & (A_{2n}) \\
      2 a_{n-1,n} k \delta_{k,-\ell} \not{c} & (D_{n+1})
    \end{cases}$ \\
 \item $[\al_n(k), \al_n(l)] =
    \begin{cases}
      2 k \delta_{k,-\ell} \not{c} & (A_{2n-1}, D_4) \\
      8 k \delta_{k,-\ell} \not{c} & (A_{2n}) \\
      4 k \delta_{k,-\ell} \not{c} & (D_{n+1})
    \end{cases}$
 \item $[\al_i(k), X(\pm \al_j, l)] =
      \pm a_{ij} X(\pm \al_j, k+l)$ \\
where $i,j \in \tilde{I}$
 \item $[X(\pm \al_i, k) , X(\pm \al_i, l) ] = 0 $ \\
where $i \in \tilde{I}$.
 \item $[X(\al_i, k) , X(-\al_j, l) ] = \\
    \begin{cases}
      \delta_{i,j} \big\{ \al_i(k+l) + (r - \delta_{i,n} (r-1)) k \delta_{k,-\ell} \not{c} \big \} & (A_{2n-1}, D_4) \\
      \delta_{i,j} \big\{ \al_i(k+l) + (r(1+\delta_{i,n}(r-1)) - \delta_{i,0}(r-1) ) k \delta_{k,-\ell} \not{c} \big \} & (A_{2n}) \\
      \delta_{i,j} \big\{ \al_i(k+l) + (1 + (\delta_{i,0}+\delta_{i,n}) (r-1)) k \delta_{k,-\ell} \not{c} \big \} & (D_{n+1})
    \end{cases}$  \\
where $i,j \in \tilde{I}$
 \item $\text{ad} X(\pm \al_i , k_2) X(\pm \al_j , k_1) = 0$ for $i,j \in \tilde{I}$ with $i \neq j$ and $a_{ij} = 0$.
 \item $\text{ad} X(\pm \al_i, k_3) \text{ad} X(\pm \al_i , k_2) X(\pm \al_j , k_1) = 0$ for $i,j \in \tilde{I}$ with $i \neq j$ and $a_{ij} = -1$
 \item $\text{ad} X(\pm \al_i, k_4) \text{ad} X(\pm \al_i , k_3) \text{ad} X(\pm \al_i , k_2) X(\pm \al_j , k_1) = 0$ for $i,j \in \tilde{I}$ with $i \neq j$ and $a_{ij} = -2$
 \item $\text{ad} X(\pm \al_i, k_5) \text{ad} X(\pm \al_i, k_4) \text{ad} X(\pm \al_i , k_3) \text{ad} X(\pm \al_i , k_2) X(\pm \al_j , k_1) = 0$ for $i,j \in \tilde{I}$ with $i \neq j$ and $a_{ij} = -3$
\end{itemize}
In addition, $\not{c}$ is central.

 Let $z,w,z_1,z_2,...$ be formal variables. We define formal power series with
 coefficients from the toroidal Lie algebra $t(\g)$:
 $$
 \al_i(z)=\sum_{n\in \Z}\al_i(n)z^{-n-1},
\qquad X(\pm \al_i,z)=\sum_{n\in \Z}X(\pm \al_i, n)z^{-n-1},
 $$
 for $i=0,1, \cdots , n$.
We will use the delta function
\begin{equation*}
\delta(z-w)=\sum_{n\in\Z}w^nz^{-n-1}
\end{equation*}

Using $\displaystyle \frac1{z-w}=\sum_{n=0}^{\infty}z^{-n-1}w^n$, $|z|>|w|$,
we have the following useful expansions:
\begin{align*}
\delta(z-w)&=\iota_{z, w}((z-w)^{-1})+\iota_{w, z}((w-z)^{-1}),\\
\partial_w\delta(z-w)&=
\iota_{z, w}((z-w)^{-2})-\iota_{w, z}((w-z)^{-2}),
\end{align*}
where $\iota_{z, w}$ means the expansion in the region $|z|>|w|$. For simplicity
in the following we will drop $\iota_{z, w}$ if it is
clear from the context.

The above defining relations of $t(\g)$ can be written in terms of power series as follows.
\begin{enumerate}\label{twrelations}
  \item $[\al_0(z), \al_0(w)] =
    \begin{cases}
      2r \partial_w\delta(z-w) \not{c} & (A_{2n-1}, D_{n+1}, D_4) \\
      2 \partial_w\delta(z-w) \not{c} & (A_{2n})
    \end{cases}$ \\
  \item $[\al_0(z), \al_j(w)] =
    \begin{cases}
      r a_{0j} \partial_w\delta(z-w) \not{c} & (A_{2n-1}, A_{2n}, D_4) \\
      a_{0j} \partial_w\delta(z-w) \not{c} & (D_{n+1})
    \end{cases}$ \\
where $j \in I$.
  \item $[\al_i(z), \al_j(w)] =
    \begin{cases}
      r a_{ij} \partial_w\delta(z-w) \not{c} & (A_{2n-1}, A_{2n}, D_4) \\
      a_{ij} \partial_w\delta(z-w) \not{c} & (D_{n+1})
    \end{cases}$ \\
where $i,j \in I$ with $i \leq j$ and $(i, j) \neq (n-1, n), (n, n)$.
 \item $[\al_{n-1}(z), \al_n(w)] =
    \begin{cases}
      a_{n-1,n} \partial_w\delta(z-w) \not{c} & (A_{2n-1}, D_4) \\
      4 a_{n-1,n} \partial_w\delta(z-w) \not{c} & (A_{2n}) \\
      2 a_{n-1,n} \partial_w\delta(z-w) \not{c} & (D_{n+1})
    \end{cases}$ \\
 \item $[\al_n(z), \al_n(w)] =
    \begin{cases}
      2 \partial_w\delta(z-w) \not{c} & (A_{2n-1}, D_4) \\
      8 \partial_w\delta(z-w) \not{c} & (A_{2n}) \\
      4 \partial_w\delta(z-w) \not{c} & (D_{n+1})
    \end{cases}$
 \item $[\al_i(z), X(\pm \al_j, w)] =
      \pm a_{ij} X(\pm \al_j, w) \delta(z-w)$ \\
where $i,j \in \tilde{I}$
 \item $[X(\pm \al_i, z) , X(\pm \al_i, w) ] = 0 $ \\
where $i \in \tilde{I}$.
 \item $[X(\al_i, z) , X(-\al_j, w) ] = \\
    \begin{cases}
      \delta_{i,j} \big\{ \al_i(w) \delta(z-w) + (r - \delta_{i,n} (r-1)) \partial_w\delta(z-w) \not{c} \big \} \qquad (A_{2n-1}, D_4) \\
      \delta_{i,j} \big\{ \al_i(w) \delta(z-w) + (r(1+\delta_{i,n}(r-1)) - \delta_{i,0}(r-1) ) \partial_w\delta(z-w) \not{c} \big \} \, (A_{2n}) \\
      \delta_{i,j} \big\{ \al_i(w) \delta(z-w) + (1 + (\delta_{i,0}+\delta_{i,n}) (r-1)) \partial_w\delta(z-w) \not{c} \big \} \qquad (D_{n+1})
    \end{cases}$  \\
where $i,j \in \tilde{I}$
 \item $\text{ad} X(\pm \al_i , z_2) X(\pm \al_j , z_1) = 0$ for $i,j \in \tilde{I}$ with $i \neq j$ and $a_{ij} = 0$.
 \item $\text{ad} X(\pm \al_i, z_3) \text{ad} X(\pm \al_i , z_2) X(\pm \al_j , z_1) = 0$ for $i,j \in \tilde{I}$ with $i \neq j$ and $a_{ij} = -1$
 \item $\text{ad} X(\pm \al_i, z_4) \text{ad} X(\pm \al_i , z_3) \text{ad} X(\pm \al_i , z_2) X(\pm \al_j , z_1) = 0$ for $i,j \in \tilde{I}$ with $i \neq j$ and $a_{ij} = -2$
 \item $\text{ad} X(\pm \al_i, z_5) \text{ad} X(\pm \al_i, z_4) \text{ad} X(\pm \al_i , z_3) \text{ad} X(\pm \al_i , z_2) X(\pm \al_j , z_1) = 0$ for $i,j \in \tilde{I}$ with $i \neq j$ and $a_{ij} = -3$
\end{enumerate}

We define a $\Z \times \hat{Q}$ grading of $L(\g, \si)$ as follows:
\begin{center}
$\text{deg}(\si^ph'_i\otimes s^k) = (k, 0)$,\\
$\text{deg}(\si^pe'_i\otimes s^k) = (k, \alpha_i)$,\\
$\text{deg}(\si^pf'_i\otimes s^k) = (k, -\alpha_i)$,\\
$\text{deg}(\si^pf'_{\theta^0}\otimes s^kt) = (k, \alpha_0)$,\\
$\text{deg}(\si^pe'_{\theta^0}\otimes s^kt^{-1}) = (k, -\alpha_0)$,\\
\end{center}
for $0\leq p\leq r-1$, $i \in I$.

Following \cite{MRY} we define the  $\Z \times \hat{Q}$ grading of $t(\g)$ as follows.
\begin{center}
$\text{deg } \not{c} := (0, 0)$ \\
$\text{deg } \al_i (k) := (k, 0)$ \\
$\text{deg } X(\pm \al_i, k) := (k, \pm \alpha_i)$
\end{center}
for $i \in \tilde{I}$ and $k \in \Z$. We define  $\hat{Q}_{\pm}:=\pm \sum_{i=0}^{n} \Z_{\geq 0} \alpha_i  \backslash \{0\}$.
Denote by $\mf{t}_{k}^{\alpha}$ the subspace of $t(\g)$ spanned by elements of degree $(k, \alpha)$. Consider the following subspaces of $t(\g)$:

%
$\displaystyle \mf{t}_{k}^{\pm}:=\sum_{\alpha \in \hat{Q}_{\pm}} \mf{t}_{k}^{\alpha},
\hspace{7mm} \displaystyle \mf{t}_{k}:=\sum_{\alpha \in \hat{Q}} \mf{t}_{k}^{\alpha}$,

$\displaystyle \mf{t}^{\alpha}:= \sum_{k \in \Z} \mf{t}_{k}^{\alpha},
\hspace{7mm} \mf{t}^{\pm}:= \sum_{k \in \Z} \mf{t}_{k}^{\pm}$.

$\displaystyle \mf{s}_{k}^{\pm}:=\text{span}\left\{ [X(\pm \al_{m_j},k_j), \ldots , X(\pm \al_{m_1},k_1)]  \mid m_i \in \tilde{I}, k_i \in \Z ,
\sum_{i=1}^{j}k_i= k\right\}$,

where $[X(\pm \al_{m_j},k_j), \ldots , X(\pm \al_{m_1},k_1)] = {\text{ad}}_ {X(\pm \al_{m_j},k_j)} \ldots {\text{ad}}_ {X(\pm \al_{m_2},k_2)}X(\pm \al_{m_1},k_1)$.


$\displaystyle \mf{s}^{\pm}:=\sum_{k \in \Z} \mf{s}_{k}^{\pm} , \, \mf{s}_{k}^{0}:=\text{span}\left\{ \delta_{k,0}\not{c}, \al_{i} (k) \mid i \in \tilde{I}, k \in \Z \right\}$,

and

$\displaystyle \mf{s}^{0}:=\sum_{k \in \Z} \mf{s}_{k}^{0}, \hspace{7mm} \mf{s}:=\mf{s}^{-}+\mf{s}^{0}+\mf{s}^{+}$.

We observe that $X(\pm \al_i, k) \in \mf{s}_{k}^{\pm}$ for each $i \in \tilde{I}$ and $\mf{s} \subset t(\g)$. The following result is an analog
of Lemma 3.1 in \cite{MRY} and follows similarly.

\begin{lemma}
\label{stlemma}
We have
\begin{enumerate}
 \item $\mf{t}_{k}^{\pm}=\mf{s}_{k}^{\pm}$, \, $\mf{t}^{\pm}=\mf{s}^{\pm}$, and $t(\g)=\mf{s}$.
 \item $\mf{t}_{k}=\mf{t}_{k}^{-}+\mf{t}_{k}^{0}+\mf{t}_{k}^{+}$ and $t(\g)=\mf{t}^{-}+\mf{t}^{0}+\mf{t}^{+}$.
\end{enumerate}
\end{lemma}

Denote by $t_{0}(\g)$ the subalgebra of $t(\g)$ generated by $\al_{i} (0), X(\pm \al_{i}, 0)$ for $i \in \tilde{I}$. Then $t_{0}(\g)$ satisfies the relations for the twisted affine algebra $\hat{\g} =A_{2n-1}^{(2)}, D_{n+1}^{(2)}, A_{2n}^{(2)}$, or $D_4^{(3)}$ and in fact $\hat{\g} \cong t_{0}(\g)$. The following result is an analog of Proposition 3.2 in \cite{MRY} which can be proved by similar argument.

\begin{lemma}
\begin{equation*}\label{dimtkalpha}
\text{dim}(\mf{t}_{k}^{\alpha})=
\begin{cases}
 1 \text{ if } \alpha \in \hat{\Delta}^{\text{re}} \\
 0 \text{ if } \alpha \notin \hat{\Delta}
\end{cases}
\end{equation*}
where $\hat{\Delta}^{\text{re}}$ is the set of real roots.
\end{lemma}

Define the map $\bar{\pi} : t(\g) \longrightarrow L(\g, \si)$ as follows:\\
In types $A_{2n-1}, D_{n+1}, D_4$,
\begin{equation*}\label{piADD}
\begin{cases}
 \not{c} \mapsto 0 \\
 \al_{0} (k) \mapsto \sum_{p=0}^{r-1} \sigma^{p} (-h'_{\theta^0}) \otimes s^{k} \\
 \al_{i} (k) \mapsto \left(1 - \delta_{i,\sigma(i)} \big(1-\frac{1}{r} \big) \right) \sum_{p=0}^{r-1} \sigma^{p} (h'_{i}) \otimes s^{k} \\
 X(\al_{0}, k) \mapsto \sum_{p=0}^{r-1} - \sigma^{p} (\omega^{r-p} f'_{\theta^0}) \otimes (s^{k} t) \\
 X(-\al_{0}, k) \mapsto \sum_{p=0}^{r-1} - \sigma^{p} (\omega^{p} e'_{\theta^0}) \otimes (s^{k} t^{-1}) \\
 X(\al_{i}, k) \mapsto \left(1 - \delta_{i,\sigma(i)} \big(1-\frac{1}{r} \big) \right) \sum_{p=0}^{r-1} \sigma^{p} (e'_{i}) \otimes s^{k} \\
 X(-\al_{i}, k) \mapsto \left(1 - \delta_{i,\sigma(i)} \big(1-\frac{1}{r} \big) \right) \sum_{p=0}^{r-1} \sigma^{p} (f'_{i}) \otimes s^{k}
\end{cases}
\end{equation*}

In type $A_{2n}$,
\begin{equation*}\label{piA2}
\begin{cases}
 \not{c} \mapsto 0 \\
 \al_{0} (k) \mapsto -h'_{\theta^0} \otimes s^{k} \\
 \al_{i} (k) \mapsto \left(1 + \delta_{i,n} \right) \sum_{p=0}^{r-1} \sigma^{p} (h'_{i}) \otimes s^{k} \\
 X(\al_{0}, k) \mapsto - f'_{\theta^0} \otimes (s^{k} t) \\
 X(-\al_{0}, k) \mapsto - e'_{\theta^0} \otimes (s^{k} t^{-1}) \\
 X(\al_{i}, k) \mapsto \left(1 + \delta_{i,n} (\sqrt{2}-1) \right) \sum_{p=0}^{r-1} \sigma^{p} (e'_{i}) \otimes s^{k} \\
 X(-\al_{i}, k) \mapsto \left(1 + \delta_{i,n} (\sqrt{2}-1) \right) \sum_{p=0}^{r-1} \sigma^{p} (f'_{i}) \otimes s^{k}
\end{cases}
\end{equation*}

The following theorem shows that $t(\g)$ is a realization of the twisted toroidal Lie algebra
$T(\g)$.

\begin{theorem}\label{MRY}
The map $\bar{\pi}$ is a surjective homomorphism, the kernel of $\bar{\pi}$ is contained in the center $Z(t(\g))$ and $(t(\g), \bar{\pi})$ is the universal central extension of $L(\g, \si)$.
\end{theorem}

\begin{proof}

To see that $\bar{\pi}$ is surjective, we observe that by (\cite{K}, Theorem 8.3) $L(\g, \si)_0 = \{ x \otimes t^{k} \, | \, k \in \Z, x \in \g_{k} \}\cong \hat{\g}$. Hence $\bar{\pi}|_{\mf{t}_{0}(\mf{g})}$ maps onto $L(\g, \si)_0$. Therefore, $\bar{\pi}$ is surjective if it is a homomorphism which we show below.

We define the map $\psi : t(\g) \rightarrow T(\g)$ as follows.

For types $A_{2n-1}, D_{n+1}, D_4$ we define $\psi$ by:
\begin{equation}\label{psiADD}
\begin{cases}
 \not{c} \mapsto \overline{s^{-1}ds} \\
 \alpha_{0} (k) \mapsto \sum_{p=0}^{r-1} \sigma^{p} (-h'_{\theta^0}) \otimes s^{k} + \overline{s^{k}t^{-1}dt} \\
 \alpha_{i} (k) \mapsto \left(1 - \delta_{i,\sigma(i)} \big(1-\frac{1}{r} \big) \right) \sum_{p=0}^{r-1} \sigma^{p} (h'_{i}) \otimes s^{k} \\
 X(\alpha_{0}, k) \mapsto \sum_{p=0}^{r-1} - \sigma^{p} (\omega^{r-j} f'_{\theta^0}) \otimes (s^{k} t) \\
 X(-\alpha_{0}, k) \mapsto \sum_{p=0}^{r-1} - \sigma^{p} (\omega^{j} e'_{\theta^0}) \otimes (s^{k} t^{-1}) \\
 X(\alpha_{i}, k) \mapsto \left(1 - \delta_{i,\sigma(i)} \big(1-\frac{1}{r} \big) \right) \sum_{p=0}^{r-1} \sigma^{p} (e'_{i}) \otimes s^{k} \\
 X(-\alpha_{i}, k) \mapsto \left(1 - \delta_{i,\sigma(i)} \big(1-\frac{1}{r} \big) \right) \sum_{p=0}^{r-1} \sigma^{p} (f'_{i}) \otimes s^{k}
\end{cases}
\end{equation}

For type $A_{2n}$ we define $\psi$ by:
\begin{equation}\label{psiA2}
\begin{cases}
 \not{c} \mapsto \overline{s^{-1}ds} \\
 \alpha_{0} (k) \mapsto -h'_{\theta^0} \otimes s^{k} + \overline{s^{k}t^{-1}dt} \\
 \alpha_{i} (k) \mapsto \left(1 + \delta_{i,n} \right) \sum_{p=0}^{r-1} \sigma^{p} (h'_{i}) \otimes s^{k} \\
 X(\alpha_{0}, k) \mapsto - f'_{\theta^0} \otimes (s^{k} t^{-1}) \\
 X(-\alpha_{0}, k) \mapsto - e'_{\theta^0} \otimes (s^{k} t^{-1}) \\
 X(\alpha_{i}, k) \mapsto \left(1 + \delta_{i,n} (\sqrt{2}-1) \right) \sum_{p=0}^{r-1} \sigma^{p} (e'_{i}) \otimes s^{k} \\
 X(-\alpha_{i}, k) \mapsto \left(1 + \delta_{i,n} (\sqrt{2}-1) \right) \sum_{p=0}^{r-1} \sigma^{p} (f'_{i}) \otimes s^{k}
\end{cases}
\end{equation}

Note that the maps $\psi$ and $\bar{\pi}$ differ only on $\not{c}$ and $\alpha_0(k)$ by elements of $\cK$. Hence $\eta \psi = \bar{\pi}$. So $\bar{\pi}$ is a homomorphism if $\psi$ is so. It suffices to show that $\psi$ preserves the defining relations which can be shown by direct calculations. For example, using (\ref{kahlercalc}) for types $A_{2n-1}, D_{n+1}, D_4$, \\
$\displaystyle \big[ \psi \big( \alpha_{0} (k) \big ), \psi \big( \alpha_{j} (l) \big) \big]$ \\
$\displaystyle = \big[ \sum_{p=0}^{r-1} \sigma^{p} (-h'_{\theta^0}) \otimes s^{k} + \overline{s^{k}t^{-1}dt} ,
 \left(1 - \delta_{j,\sigma(j)} \big(1-\frac{1}{r} \big) \right) \sum_{q=0}^{r-1} \sigma^{q} (h'_{j}) \otimes s^{l} \big]$ \\
$\displaystyle = \left(1 - \delta_{j,\sigma(j)} \big(1-\frac{1}{r} \big) \right) \sum_{p,q=0}^{r-1} ( \sigma^{p} (-h'_{\theta^0}) | \sigma^{q} (h'_{j}) ) \overline{s^{l}ds^{k}}$. \\
By direct computation we have

$( \sigma^{p} (-h'_{\theta^0}) | \sigma^{q} (h'_{j}) ) =
 \begin{cases}
  -\delta_{1j} & p=q, (A_{2n-1}) \\
  \delta_{1j} - \delta_{2j} & p \neq q, (A_{2n-1}) \\
  -\delta_{1j} - \delta_{nj} & p=q, (D_{n+1}) \\
  -\delta_{1j} + \delta_{nj} & p \neq q, (D_{n+1}) \\
  -\delta_{1j} & p=q, (D_4) \\
  -\delta_{1j} & p \equiv q-1 \text{ (mod 3)}, (D_4) \\
  \delta_{1j} & p \equiv q-2 \text{ (mod 3)}, (D_4)
 \end{cases}$. \\
Hence, in type $A_{2n-1}$ we have \\
$\displaystyle \big[ \psi \big( \alpha_{0} (k) \big ), \psi \big( \alpha_{j} (l) \big) \big] = \left(1 - \delta_{j,\sigma(j)} \big(1-\frac{1}{r} \big) \right) r(-\delta_{1j} + \delta_{1j} - \delta_{2j}) \overline{s^{l}ds^{k}}$ \\
$\displaystyle = -r\delta_{2j} k \delta_{k,-l} \psi (\not{c}) = r a_{0j} k \delta_{k,-l} \psi (\not{c})$. \\
In type $D_{n+1}$ we have \\
$\displaystyle \big[ \psi \big( \alpha_{0} (k) \big ), \psi \big( \alpha_{j} (l) \big) \big] = \left(1 - \delta_{j,\sigma(j)} \big(1-\frac{1}{r} \big) \right) r(-\delta_{1j} - \delta_{nj} -\delta_{1j} + \delta_{nj}) \overline{s^{l}ds^{k}}$ \\
$\displaystyle = \frac{1}{r} r(-r\delta_{1j}) k \delta_{k,-l} \psi (\not{c}) = a_{0j} k \delta_{k,-l} \psi (\not{c})$. \\
In type $D_4$ we have \\
$\displaystyle \big[ \psi \big( \alpha_{0} (k) \big ), \psi \big( \alpha_{j} (l) \big) \big] = \left(1 - \delta_{j,\sigma(j)} \big(1-\frac{1}{r} \big) \right) r (-\delta_{1j} -  \delta_{1j} + \delta_{1j}) \overline{s^{l}ds^{k}}$ \\
$\displaystyle = r(-\delta_{1j}) k \delta_{k,-l} \psi (\not{c}) = r a_{0j} k \delta_{k,-l} \psi (\not{c})$. \\
For type $A_{2n}$, \\
$\displaystyle \big[ \psi \big( \alpha_{0} (k) \big ), \psi \big( \alpha_{j} (l) \big) \big]
= \big[ -h'_{\theta^0} \otimes s^{k} + \overline{s^{k}t^{-1}dt} ,
 \left(1 + \delta_{j,n} \right) \sum_{q=0}^{r-1} \sigma^{q} (h'_{j}) \otimes s^{l} \big]$ \\
$\displaystyle = \left(1 + \delta_{j,n} \right) \sum_{q=0}^{r-1} ( -h'_{\theta^0} | \sigma^{q} (h'_{j}) ) \overline{s^{l}ds^{k}}
= -r \delta_{1j} \overline{s^{l}ds^{k}} = r a_{0j} k \delta_{k,-l} \psi (\not{c})$.

Other calculations are similar. Thus  $\psi$ is a homomorphism. Hence $\bar{\pi}$ is a surjective homomorphism. Indeed, by definition of the $\Z \times \hat{Q}$ grading of $t(\g)$ and $L(\g, \si)$ we see that $\bar{\pi}$ is a graded homomorphism. By Lemma \ref{dimtkalpha}, $\bar{\pi} (\mf{t}^{\alpha}) \neq \{ 0 \}$ if $\alpha \in \hat{\Delta}^{\text{re}}$. Thus, $\bar{\pi} (\mf{t}^{\alpha}) = \{ 0 \}$ implies that $\alpha \in \hat{\Delta}^{\text{im}}$ or $\alpha \notin \hat{\Delta}$ where $\hat{\Delta}^{\text{im}}$ is the set of imaginary roots of $\hat{\g}$. Then since any imaginary root is an integer multiple of $\delta$, we have $\text{ker } \bar{\pi} \subset \sum_{j \in \Z} \mf{t}^{j \delta}$. As $\text{ker } \bar{\pi}$ is an ideal of $t(\g)$ and $[X(\pm \al_{i}, k) , \sum_{j \in \Z} \mf{t}^{j \delta}] \cap \sum_{j \in \Z} \mf{t}^{j \delta} = \{ 0 \}$ for all $i \in \tilde{I}$, we have $\text{ker } \bar{\pi} \subset Z(t(\g))$.

It is left to show that $(t(\g), \bar{\pi})$ is the universal central extension of $L(\g, \si)$. Suppose $(\mathcal{V} , \gamma)$ is a central extension of $L(\g , \si)$. Since $(T(\g), \eta)$ is the universal central extension of $L(\g , \si)$,  we have a unique map $\lambda : T(\g) \rightarrow \mathcal{V}$ such that $\gamma \lambda = \eta$. Now we have a homomorphism $\lambda\psi : t(\g) \longrightarrow \mathcal{V}$ and
$\gamma\lambda\psi = \eta\psi = \bar{\pi}$ giving the following commuting diagram:

\begin{center}
\begin{tikzpicture}[node distance=2cm, auto]
  \node (tauh) {$T(\g)$};
  \node (drinfeldt) [below of=tauh] {$t(\g)$};
  \node (V) [above of=tauh] {$\mathcal{V}$};
  \node (L) [right of=tauh] {$L(\g , \si)$};
  \draw[->] (tauh) to node {$\lambda$} (V);
  \draw[->] (tauh) to node {$\eta$} (L);
  \draw[->] (V) to node {$\gamma$} (L);
  \draw[->] (drinfeldt) to node {$\psi$} (tauh);
  \draw[->] (drinfeldt) to node [swap] {$\bar{\pi}$} (L);
\end{tikzpicture}
\end{center}

Since $T(\g)$ is the universal central extension of $L(\g, \si)$, the lower triangle commutes which implies that the map $\psi$ is unique
and proves that $(t(\g), \bar{\pi})$ is the universal central extension of $L(\g, \si)$.

\end{proof}

\section{Fermionic Representations}

In this section we use the MRY presentation of the twisted toroidal Lie algebra $T(\g)$ in the previous section and give a fermionic free field realization of $T(\g)$ for $\g = A_{2n-1}, D_{n+1}, A_{2n}$, and $D_4$.

We consider the vector space  $\C^{n+2}$
with the standard basis $\{\vep_i \mid i = 0,1, \cdots , n+1\}$. This is an orthonormal basis with respect to the inner product $(\ | \ )$ given by
\begin{center}
$(\vep_i | \vep_j) = \delta_{ij}$.
\end{center}
Consider the lattice $P_0= \Z\vep_1 \oplus \Z\vep_2 \oplus \cdots \oplus\Z\vep_n$ and set $c = \frac{1}{\sqrt{2}} (\varepsilon_0 + i \varepsilon_{n+1})$ and $d = \frac{1}{\sqrt{2}} (\varepsilon_0 - i \varepsilon_{n+1})$. Then $( c | c) = 0 = ( d | d)$ and $( c | d) = 1$.  The simple roots of the fixed point subalgebra $\g_0$ of $\g$ can be realized as follows.

\begin{itemize}
\item $\al_i = \frac{1}{\sqrt{2}}(\vep_i - \vep_{i+1}), 1 \leq i \leq n-1,  \al_n=\sqrt{2}\varepsilon_n$, for $(\g = A_{2n-1}, r=2)$;
\item $\al_i = \vep_i - \vep_{i+1}, 1 \leq i \leq n-1, \al_n=\vep_n$, for $(\g = D_{n+1}, r=2)$;
\item $\al_i = \frac{1}{\sqrt{2}}(\vep_i - \vep_{i+1}), 1 \leq i \leq n-1, \alpha_n=\frac{1}{\sqrt{2}}\vep_n$, for $(\g = A_{2n}, r=2)$;
\item $\al_1 = \frac{1}{\sqrt{3}}(\vep_1 - \vep_2), \alpha_2 = \frac{1}{\sqrt{3}}(-\vep_1 + 2\vep_2 - \vep_3)$ for $(\g = D_4, r=3)$.
\end{itemize}

Recall that ${\g}={\g}_0\oplus \cdots\oplus{\g}_{r-1},$ and $\g_1$ is an irreducible $\g_0$-module with highest weight

\begin{center}
$\displaystyle \theta_0 := \frac{1}{r}\sum_{j=0}^{r-1} \sigma^j (\theta^0) =
\begin{cases}
 \frac{1}{\sqrt{2}}\vep_1 + \frac{1}{\sqrt{2}}\vep_2& (A_{2n-1}); \\
 \vep_1 & (D_{n+1}); \\
 \sqrt{2}\vep_1 & (A_{2n}); \\
 \frac{1}{\sqrt{3}}\vep_1- \frac{1}{\sqrt{3}}\vep_3 & (D_4).
\end{cases}$
\end{center}
We define
\begin{center}
$\displaystyle \be :=
\begin{cases}
 -\sqrt{2}c +\vep_1  & (A_{2n-1}); \\
 -c+\vep_1 & (D_{n+1}); \\
 -\frac{1}{\sqrt{2}}c + \vep_1 & (A_{2n}); \\
 -\sqrt{3}c + \vep_1& (D_4).
\end{cases}$
\end{center}
Then set
\begin{center}
$\alpha_0 := c - \theta_0 =
\begin{cases}
  -\frac{1}{\sqrt{2}}(\beta+\vep_2) & (A_{2n-1}); \\
  -\beta & (D_{n+1}); \\
  -\sqrt{2}\beta  & (A_{2n}); \\
  - \frac{1}{\sqrt{3}}(\beta -\vep_3)  & (D_4).
\end{cases}$
\end{center}
It is easy to verify that $\{\a_i | 0\leq i \leq n\}$ form the set of simple roots of the twisted affine Lie algebra $\hat{\g}$ with the GCM
$\tilde{A} = (a_{ij})_{i,j \in \tilde{I}}$ and the nondegenerate invariant bilinear form given by
\begin{equation}
(\alpha_i|\alpha_j)=d_ia_{ij},
\ \ \mbox{for all} \ i,j\in \tilde{I},
\end{equation}
where
\begin{center}
$\displaystyle (d_0,d_1, \cdots , d_n) =
\begin{cases}
 (\frac{1}{2},\frac{1}{2}, \cdots , \frac{1}{2}, 1)& (A_{2n-1}); \\
 (\frac{1}{2}, 1, \cdots, 1, \frac{1}{2}) & (D_{n+1}); \\
 (1, \frac{1}{2}, \cdots , \frac{1}{2}, \frac{1}{4}) & (A_{2n}); \\
 (\frac{1}{3}, \frac{1}{3}, 1) & (D_4).
\end{cases}$
\end{center}
Furthermore, we observe that $(c | \al_i) = 0 = (d | \al_i), i \in \tilde{I}$  and $c$ (resp. $d$) corresponds to the null root $\delta$ (resp. dual gradation operator)
for $\hat{\g}$.

Now we consider the lattice
\begin{equation*}
P = P_0 \oplus \bar{P}_0\oplus \Z c
\end{equation*}
where $\bar{P}_0 = \oplus_{i=1}^{n}\Z \vep_{\bar{i}}$ and $(\vep_{\bar{i}} | \vep_{\bar{j}}) = \delta_{ij}, (c | \vep_{\bar{i}}) = (\vep_i | \vep_{\bar{j}}) =  0$. Let $P_{\C} = \C \otimes P$ be the $\C$-vector space spanned by $\{c, \vep_i, \vep_{\bar{i}} \mid 1 \leq i \leq n\}$. Then we define the vector space  $\mathcal C= P_{\mathbb C}\oplus P_{\mathbb C}^*$,
where both subspaces  $P_{\mathbb C}$ and
$P_{\mathbb C}^*$ are maximal isotropic subspaces
with the symmetric bilinear form on $\mathcal C$
given by
\begin{equation}
<b^*, a>=<a, b^*>=(a|b), \quad <a, b>=<a^*, b^*>=0,
\end{equation}
for all $a, b\in P_{\mathbb C}$.
Thus we have a maximal polarization of $\mathcal C$.

We consider the Clifford algebra $Cl(P)$ generated by the central element $\bf 1$ and elements $a(k), a^*(k)$ where $a \in P_{\C}, k \in \Za$
subject to the relations:
$$
\{a(k), b(l)\}=\{a^*(k), b^*(l)\}=0 , \{a(k), b^*(l)\}=(a|b)\delta_{k, -l}{\bf 1}
$$
where $a, b\in P_{\C}$.

We consider the representation space to be the infinite dimensional vector space
\begin{center}
$\displaystyle V := \bigotimes_{a} \left( \bigotimes_{k \in \Zp} \C [a(-k)] \bigotimes_{k \in \Zp} \C [a^* (-k)] \right)$
\end{center}
where the $a \in \{c, \vep_i, \vep_{\bar{i}} \mid 1 \leq i \leq n\}$. The Clifford algebra acts on $V$ by the following action: for $k \in \Zp$,
$a(-k)$ acts as a creation operator, $a(k)$ acts as an annihilation operator and ${\bf 1}$ as the identity. For any two fermionic fields

$$a(z) = \dsum_{m \in \Za}a(m)z^{-m-1/2} \ \  {\mbox{and}}  \ \  b(w) = \dsum_{n \in \Za}b(n)w^{-n-1/2}$$
we define the normal ordering $:a(z)b(w):$ by their components:
\begin{center}
$: \! a(m) b(n) \! : =
\begin{cases}
a(m)b(n), & \mbox{if } m < 0; \\
-b(n)a(m), & \mbox{if } m > 0.
\end{cases}$
\end{center}
Hence the normal ordering satisfies the relation
$$
:a(z)b(w): = - :b(w)a(z):,
$$
which implies $:a(z)a(z): = 0$. Hence $:c(z)\vep_i(w): = 0 = :c(z)\vep_{\bar{i}}(w):$ since $c$  is central. The normal product of $k$ fields is defined inductively by:
$$
:a_1(z_1)a_2(z_2) \cdots a_k(z_k): = :a_1(z_1)(:a_2(z_2) \cdots a_k(z_k):):.
$$
We define the contraction of two states by
$$
\underbrace{a(z)b(w)}=a(z)b(w)-:a(z)b(w):,
$$
which contains all poles for $a(z)b(w)$. In particular we have the following (see \cite{JM}, Proposition 3.1)
\begin{align*}
\underbrace{a(z)b(w)}&=\underbrace{a^*(z)b^*(w)}=0,\\
\underbrace{a(z)b^*(w)}&=\underbrace{a^*(z)b(w)}
=\frac{(a, b)}{z-w}.
\end{align*}
Hence the fermionic fields satisfy the following anticommutation relations (see \cite{JM}, Proposition 3.2):
\begin{align*}
\{a(z), b(w)\}&=\{a^*(z), b^*(w)\}=0,\\
\{a(z), b^*(w)\}&=(a, b)\delta(z-w).
\end{align*}

We use the following result to calculate the bracket among normal order products which can be obtained by Wick's theorem (see \cite{K1}, Theorem 3.3).
\begin{proposition}(\cite{JM}, Proposition 3.3)
\label{fermbracket}
For $a_1, b_1, a_2, b_2 \in \mathcal{C}$ and formal variables $z,w$, we have

$[ : \! a_1 (z) b_1 (z) \! : , : \! a_2 (w) b_2 (w) \! : ] $ \\
$ = \langle a_1 , b_2 \rangle : \! b_1 (w) a_2 (w) \! : \delta (z - w)
 - \langle a_1 , a_2 \rangle : \! b_1 (w) b_2 (w) \! : \delta (z - w)$ \\
$ + \langle b_1 , a_2 \rangle : \! a_1 (w) b_2 (w) \! : \delta (z - w)
 - \langle b_1 , b_2 \rangle : \! a_1 (w) a_2 (w) \! : \delta (z - w)$ \\
$ + (\langle a_1 , b_2 \rangle  \langle b_1 , a_2 \rangle -  \langle a_1 , a_2 \rangle  \langle b_1 , b_2 \rangle) \partial_w \delta (z - w)$.
\end{proposition}

For $\g = D_{n+1}$ (resp. $\g = A_{2n}$) we define two ghost fields
$$
\displaystyle e(z) = \sum_{k \in \Za} e(k) z^{-k-\frac{1}{2}} \ \ \mbox{and} \ \ \displaystyle \overline{e}(z) = \sum_{k \in \Za} \overline{e}(k) z^{-k-\frac{1}{2}}
$$
which have the only nonzero symmetric bilinear products
$\langle e , e \rangle = 1 = \langle \overline{e} , \overline{e} \rangle$ (resp. $\langle e , \overline{e} \rangle = 1 = \langle \overline{e}, e \rangle$).

In the following theorem, using the MRY presentation in Theorem \ref{MRY} we give a level $(1,0)$ fermionic representation of the twisted toroidal algebra $t(\g)$, $(\g = A_{2n-1}, D_{n+1}, A_{2n}, D_4)$ on $V$.

\begin{theorem}
Under the following map we have a level $(1,0)$ representation of the twisted toroidal Lie algebra $t(\g)$ on $V$:
 \begin{align*}
 X(\al_0,z) &= \begin{cases}:\vep_{\ov{2}}(z)\be^*(z): + :\vep_2^*(z)\vep_{\ov{1}}(z):& (A_{2n-1}),\\
 \sqrt{2}:\ov{e}(z)\be^*(z):& (D_{n+1}),\\
 :\vep_{\ov{1}}(z)\be^*(z):& (A_{2n}),\\
 :\vep_{\ov{1}}(z)\be^*(z): + \omega :\vep_{\ov{1}}^*(z)\vep_1^*(z): +\omega^2 :\vep_{\ov{2}}^*(z)\vep_2^*(z):& (D_4),
\end{cases} \end{align*}
\begin{align*}
 X(-\al_0,z) &= \begin{cases}:\be(z)\vep_{\ov{2}}^*(z): + :\vep_{\ov{1}}^*(z)\vep_2(z):& (A_{2n-1}),\\
 \sqrt{2}:\be(z)\ov{e}(z):& (D_{n+1}),\\
 :\be(z)\vep_{\ov{1}}^*(z):& (A_{2n}),\\
 :\be(z)\vep_{\ov{1}}^*(z): + \omega^2 :\vep_1(z)\vep_{\ov{1}}(z): +\omega :\vep_2(z)\vep_{\ov{2}}(z):& (D_4),
\end{cases}
 \end{align*}
 \begin{align*}
 X(\al_i,z)&= \begin{cases}:\vep_i(z) \vep^*_{i+1}(z): + :\vep_{\ov{i}}^*(z) \vep_{\ov{i+1}}(z):& (A_{2n-1}),\\
 :\vep_i(z) \vep_{i+1}^*(z):& (D_{n+1}),\\
 :\vep_i(z) \vep_{i+1}^*(z): + :\vep_{\ov{i}}^*(z) \vep_{\ov{i+1}}(z):& (A_{2n}),\\
 :\vep_1(z) \vep_2^*(z): + :\vep_{\ov{2}}(z) \vep_{\ov{1}}^*(z): + :\vep_{\ov{2}}(z) \vep_{\ov{1}}(z):& (D_4, i=1),
 \end{cases}
 \end{align*}
  \begin{align*}
 X(-\al_i,z)&= \begin{cases}:\vep_{i+1}(z)\vep_i^*(z): + :\vep_{\ov{i+1}}^*(z) \vep_{\ov{i}}(z):& (A_{2n-1}),\\
 :\vep_{i+1}(z) \vep_{i}^*(z):& (D_{n+1}),\\
 :\vep_{i+1}(z) \vep_{i}^*(z): + :\vep_{\ov{i+1}}^*(z) \vep_{\ov{i}}(z):& (A_{2n}),\\
 :\vep_2(z) \vep_1^*(z): + :\vep_{\ov{1}}(z) \vep_{\ov{2}}^*(z): + :\vep_{\ov{1}}^*(z) \vep_{\ov{2}}^*(z):& (D_4, i=1),
 \end{cases}
 \end{align*}

 for $1 \leq i \leq n-1$.
 \begin{align*}
 X(\al_n,z)&= \begin{cases} :\vep_n(z) \vep^*_{\overline{n}}(z) \! & (A_{2n-1}),\\
 \sqrt{2} :\vep_n(z) e(z): & (D_{n+1}),\\
\sqrt{2}(:\vep_n(z)\ov{e}(z): +:\vep_{\ov{n}}^*(z)e(z):)& (A_{2n}),\\
:\vep_2(z) \vep^*_{\ov{2}}(z) \! :  & (D_4, n=2),
 \end{cases}
 \end{align*}
\begin{align*}
 X(-\al_n,z)&= \begin{cases} :\vep_{\ov{n}}(z) \vep^*_n(z): & (A_{2n-1}),\\
 \sqrt{2} :e(z)\vep_n^*(z): & (D_{n+1}),\\
 \sqrt{2} (:e(z)\vep^*_{n}(z): + :\ov{e}(z)\vep_{\ov{n}}(z):) & (A_{2n}),\\
:\vep_{\ov{2}}(z) \vep_2^*(z):  & (D_4, n=2).
 \end{cases}
 \end{align*}
 The fields for the simple roots are represented by:
 \begin{align*}
 \al_0(z) &= \begin{cases}:\be^*(z) \be(z): + :\vep^*_2(z)\vep_2(z): + :\vep_{\ov{1}}(z) \vep^*_{\ov{1}}(z): + :\vep_{\ov{2}}(z) \vep^*_{\ov{2}}(z):& (A_{2n-1}),\\
  2 :\be^*(z) \be(z):& (D_{n+1}),\\
 :\vep_{\ov{1}}(z) \vep^*_{\ov{1}}(z): + : \be^*(z) \be(z):& (A_{2n}),\\
 :\vep^*_1(z) \vep_1(z): + :\be^*(z) \be(z):+ :\vep^*_2(z) \vep_2(z): + :\vep^*_{\ov{2}}(z)\vep_{\ov{2}}(z):& (D_4),
 \end{cases}
 \end{align*}
 \begin{align*}
 \al_i(z) &= \begin{cases} :\vep_i(z) \vep^*_i(z): + :\vep_{i+1}^*(z) \vep_{i+1}(z): + : \vep_{\ov{i}}^*(z) \vep_{\ov{i}}(z): + :\vep_{\ov{i+1}}(z) \vep^*_{\ov{i+1}}(z):& (A_{2n-1}),\\
 :\vep_i(z) \vep^*_i(z): + :\vep_{i+1}^*(z) \vep_{i+1}(z):& (D_{n+1}),\\
 :\vep_i(z)\vep_i^*(z): + : \! \vep_{i+1}^*(z) \vep_{i+1}(z): + :\vep_{\ov{i+1}}(z)\vep_{\ov{i+1}}^*(z): + :\vep_{\ov{i}}^*(z) \vep_{\ov{i}}(z): & (A_{2n}),\\
 :\vep_1(z)\vep_1^*(z): + :\vep_2^*(z)\vep_2(z): + 2:\vep_{\ov{2}}(z)\vep_{\ov{2}}^*(z): & (D_4, i=1),
 \end{cases}
 \end{align*}
 for $1\leq i \leq n-1$ and
 \begin{align*}
 \al_n(z) &= \begin{cases}:\vep_n(z)\vep_n^*(z): + :\vep_{\ov{n}}^*(z)\vep_{\ov{n}}(z):& (A_{2n-1}),\\
 2 :\vep_n(z)\vep_n^*(z): & (D_{n+1}),\\
 2 (:\vep_n(z)\vep_n^*(z): + :\vep_{\ov{n}}^*(z) \vep_{\ov{n}}(z):) & (A_{2n}),\\
 :\vep_2(z)\vep^*_2(z): + :\vep_{\ov{2}}^*(z)\vep_{\ov{2}}(z): & (D_4, n=2).
 \end{cases}
 \end{align*}
\end{theorem}
\begin{proof}
It is sufficient to show that the relations $(1) - (12)$ of the MRY presentation of $t(\g)$ in Section 3 hold. In order to calculate the corresponding brackets we use Proposition \ref{fermbracket} repeatedly. Since most of these calculations are similar we only verify relations: $(6), i=0=j, \ \ i=n-1, j=n$, and $(8), i=0=j, \ \ i=n=j$.  The remaining relations can be checked similarly.

First we consider the case when $\g = A_{2n-1}$. In this case using Proposition \ref{fermbracket} we get:
\begin{align*}
[\al_0(z), X(\al_0,w)] =& [:\be^*(z)\be(z): + :\vep_2^*(z)\vep_2(z): + :\vep_{\ov{1}}(z)\vep_{\ov{1}}^*(z):\\
&+ :\vep_{\ov{2}}(z)\vep_{\ov{2}}^*(z):, :\vep_{\ov{2}}(w)\be^*(w): + :\vep_2^*(w)\vep_{\ov{1}}(w):]\\
& = -:\be^*(w)\vep_{\ov{2}}(w): \d (z-w) + :\vep_2^*(w) \vep_{\ov{1}}(w): \d (z-w) \\
& - :\vep_{\ov{1}}(w)\vep_2^*(w): \d (z-w) + :\vep_{\ov{2}}(w)\be^*(w): \d (z-w)\\
&= 2(:\vep_{\ov{2}}(w)\be^*(w): + :\vep_2^*(w) \vep_{\ov{1}}(w):) \d (z-w) \\
&= 2X(\al_0,w) \d (z-w).
\end{align*}
\begin{align*}
[\al_0(z), X(-\al_0,w)] =&[:\be^*(z)\be(z): + :\vep_2^*(z)\vep_2(z): + :\vep_{\ov{1}}(z)\vep_{\ov{1}}^*(z):\\
&+ :\vep_{\ov{2}}(z)\vep_{\ov{2}}^*(z):, :\be(w)\vep_{\ov{2}}^*(w): + :\vep_{\ov{1}}^*(w)\vep_2(w):]\\
&= - :\be(w)\vep_{\ov{2}}^*(w): \d (z-w) + :\vep_2(w) \vep_{\ov{1}}^*(w): \d (z-w) \\
& - :\vep_{\ov{1}}^*(w)\vep_2(w): \d (z-w) + :\vep_{\ov{2}}^*(w) \be(w) \! : \d (z-w) \\
&= - 2(:\be(w)\vep_{\ov{2}}^*(w): + :\vep_{\ov{1}}^*(w)\vep_2(w):) \d (z-w) \\
&= - 2X(-\al_0,w) \d (z-w).
\end{align*}
\begin{align*}
[\al_{n-1}(z), X(\al_n,w)] =&[:\vep_{n-1}(z)\vep_{n-1}^*(z): + :\vep_n^*(z)\vep_n(z): \\
&+ :\vep_{\ov{n-1}}^*(z)\vep_{\ov{n-1}}(z): + :\vep_{\ov{n}}(z)\vep_{\ov{n}}^*(z):, :\vep_n(w)\vep_{\ov{n}}^*(w):]\\
&= - :\vep_n(w)\vep_{\ov{n}}^*(w): \d (z-w) + :\vep_{\ov{n}}^*(w)\vep_n(w): \d (z-w)\\
&= -2X(\al_n,w) \d (z-w) = a_{n-1,n} X(\al_n,w) \d (z-w).
\end{align*}
\begin{align*}
[\al_{n-1}(z), X(-\al_n,w)] =& [:\vep_{n-1}(z)\vep_{n-1}^*(z): + :\vep_n^*(z)\vep_n(z): \\
&+ :\vep_{\ov{n-1}}^*(z)\vep_{\ov{n-1}}(z): + :\vep_{\ov{n}}(z)\vep_{\ov{n}}^*(z):, :\vep_{\ov{n}}(w)\vep_n^*(w):]\\
&= - :\vep_n^*(w)\vep_{\ov{n}}(w): \d (z-w) + :\vep_{\ov{n}}(w)\vep_n^*(w): \d (z-w)\\
&= 2X(-\al_n,w) \d (z-w) = - a_{n-1,n} X(\al_n,w) \d (z-w).
\end{align*}
\begin{align*}
[X(\al_0, z), X(-\al_0, w)] =&[:\vep_{\ov{2}}(z)\be^*(z): + :\vep_2^*(z)\vep_{\ov{1}}(z):,\\
& :\be(w)\vep_{\ov{2}}^*(w): + :\vep_{\ov{1}}^*(w)\vep_2(w):]\\
& :\vep_{\ov{2}}(w)\vep_{\ov{2}}^*(w): \d (z-w) + :\be^*(w) \be(w): \d (z-w) \\
&+ \partial_w \d (z-w) + :\vep_{\ov{1}}(w)\vep_{\ov{1}}^*(w): \d (z-w) \\
&+ :\vep_2^*(w)\vep_2(w): \d (z-w) + \partial_w \d (z-w)\\
&=\al_0(w)\d (z-w) + 2\partial_w\d (z-w) \\
&= \al_0(w)\d (z-w) + r\partial_w\d (z-w)\not{c},
\end{align*}
since $r = 2$ here and $\not{c}$ acts as $1$.
\begin{align*}
[X(\al_n, z), X(-\al_n, w)] =& [:\vep_n(z)\vep_{\ov{n}}^*(z):, :\vep_{\ov{n}}(w)\vep_n^*(w):]\\
&=:\vep_n(w)\vep_n^*(w): \d (z-w)+ :\vep_{\ov{n}}^*(w)\vep_{\ov{n}}(w): \d (z-w)\\
& + \partial_w \d (z-w)\\
&=\al_n(w)\d (z-w) + \partial_w\d (z-w)\not{c}.
\end{align*}
Next we consider the case when $\g = D_{n+1}$. In this case using Proposition \ref{fermbracket} we get:
\begin{align*}
[\al_0(z), X(\al_0,w)] =& [ 2 :\be^*(z)\be(z):, \sqrt{2} :\ov{e}(w)\be^*(w): ]\\
&= 2 \sqrt{2} :\ov{e}(w)\be^*(w): \d (z-w) \\
&= 2X(\al_0,w)\d (z-w).
\end{align*}
\begin{align*}
[\al_0(z), X(-\al_0,w)] =& [ 2 :\be^*(z)\be(z):, \sqrt{2} :\be(w)\ov{e}(w): ]\\
&= -2 \sqrt{2} :\be(w)\ov{e}(w): \d (z-w)\\
&= -2X(-\al_0,w)\d (z-w).
\end{align*}
\begin{align*}
[\al_{n-1}(z), X(\al_n,w)] =& [:\vep_{n-1}(z)\vep_{n-1}^*(z): + :\vep_n^*(z)\vep_n(z):, \sqrt{2}:\vep_n(w)e(w):]\\
&= -\sqrt{2}:\vep_n(w)e(w):\d (z-w) \\
&=a_{n-1,n}X(\al_n,w)\d (z-w).
\end{align*}
\begin{align*}
[\al_{n-1}(z), X(-\al_n,w)] =& [:\vep_{n-1}(z)\vep_{n-1}^*(z): + :\vep_n^*(z)\vep_n(z):, \sqrt{2}
:e(w)\vep_n^*(w):]\\
&= \sqrt{2}:e(w)\vep_n^*(w):\d (z-w) \\
&=- a_{n-1,n}X(-\al_n,w)\d (z-w).
\end{align*}
\begin{align*}
[X(\al_0, z), X(-\al_0, w)] =&[\sqrt{2}:\ov{e}(z)\be^*(z):, \sqrt{2}:\be(w)\ov{e}(w):]\\
&= 2:\be^*(w)\be(w):\d (z-w) + 2 \partial_w\d (z-w)\\
&=\al_0(w)\d (z-w) + 2\partial_w\d (z-w) \\
&= \al_0(w)\d (z-w) + r\partial_w\d (z-w)\not{c},
\end{align*}
since $r = 2$ here and $\not{c}$ acts as $1$.
\begin{align*}
[X(\al_n, z), X(-\al_n, w)] =& [\sqrt{2}:\vep_n(z)e(z):, \sqrt{2}:e(w)\vep_n^*(w):]\\
&=2:\vep_n(w)\vep_n^*(w): \d (z-w)+ 2 \partial_w \d (z-w)\\
&=\al_n(w)\d (z-w) + 2\partial_w\d (z-w)\not{c}.
\end{align*}
Next we consider the case when $\g = A_{2n}$. As before using Proposition \ref{fermbracket} we have:
\begin{align*}
[\al_0(z), X(\al_0,w)] =&[:\vep_{\ov{1}}(z)\vep_{\ov{1}}^*(z):+:\be^*(z)\be(z):, :\vep_{\ov{1}}(w)\be^*(w):]\\
&= :\vep_{\ov{1}}(w)\be^*(w): \d (z-w) - :\be^*(w)\vep_{\ov{1}}(w): \d (z-w)\\
&= 2X(\al_0,w)\d (z-w).
\end{align*}
\begin{align*}
[\al_0(z), X(-\al_0,w)] =&[:\vep_{\ov{1}}(z)\vep_{\ov{1}}^*(z):+:\be^*(z)\be(z):,
:\be(w)\vep_{\ov{1}}^*(w):]\\
&= :\vep_{\ov{1}}^*(w)\be(w): \d (z-w) - :\be(w)\vep_{\ov{1}}^*(w): \d (z-w)\\
&= -2X(-\al_0,w)\d (z-w).
\end{align*}
\begin{align*}
[\al_{n-1}(z), X(\al_n,w)] =& [:\vep_{n-1}(z)\vep_{n-1}^*(z): + :\vep_n^*(z)\vep_n(z):+:\vep_{\ov{n}}(z)\vep_{\ov{n}}^*(z):\\
&+:\vep_{\ov{n-1}}^*(z)\vep_{\ov{n-1}}(z):, \sqrt{2}(:\vep_n(w)\ov{e}(w):+:\vep_{\ov{n}}^*(w)e(w):]\\
&= \sqrt{2} (-:\vep_n(w)\ov{e}(w): - :\vep^*_{\ov{n}}(w)e(w): )\d (z-w) \\
&=a_{n-1,n}X(\al_n,w)\d (z-w).
\end{align*}
\begin{align*}
[\al_{n-1}(z), X(-\al_n,w)] =& [:\vep_{n-1}(z)\vep_{n-1}^*(z): + :\vep_n^*(z)\vep_n(z):+:\vep_{\ov{n}}(z)\vep_{\ov{n}}^*(z):\\
&+:\vep_{\ov{n-1}}^*(z)\vep_{\ov{n-1}}(z):, \sqrt{2}(:e(w)\vep_n^*(w):+:\ov{e}(w)\vep_{\ov{n}}(w):]\\
&= \sqrt{2} (-:\vep_n^*(w)e(w): - :\vep_{\ov{n}}(w)\ov{e}(w): )\d (z-w) \\
&=- a_{n-1,n}X(-\al_n,w)\d (z-w).
\end{align*}
\begin{align*}
[X(\al_0, z), X(-\al_0, w)] =&[:\vep_{\ov{1}}(z)\be^*(z):,:\be(w)\vep_{\ov{1}}^*(w):]\\
&= :\be^*(w)\be(w):\d (z-w)+:\vep_{\ov{1}}(w)\vep_{\ov{1}}^*(w): \d (z-w)\\
&+ \partial_w\d (z-w)\\
&= \al_0(w)\d (z-w) + \partial_w\d (z-w)\not{c}.
\end{align*}
\begin{align*}
[X(\al_n, z), X(-\al_n, w)] =& [\sqrt{2}(:\vep_n(z)\ov{e}(z):+:\vep_{\ov{n}}^*(z)e(z):,\\
&\sqrt{2}(:e(w)\vep_n^*(w):+:\ov{e}(w)\vep_{\ov{n}}(w):]\\
&=2(:\vep_n(w)\vep_n^*(w):+:\ov{e}(w)e(w):+:e(w)\ov{e}(w):\\
&+:\vep_{\ov{n}}^*(w)\vep_{\ov{n}}(w):)\d (z-w)+ 4 \partial_w \d (z-w)\\
&=\al_n(w)\d (z-w) + r^2 \partial_w\d (z-w)\not{c},
\end{align*}
since $:\ov{e}(w)e(w):+:e(w)\ov{e}(w): = 0$ and $r = 2$.

Finally we consider the case when $\g = D_4$. Note that in this case $r = 3$ and $n = 2$. As before using Proposition \ref{fermbracket} we have:
\begin{align*}
[\al_0(z), X(\al_0,w)] =&[:\vep^*_1(z) \vep_1(z): + :\be^*(z) \be(z):+ :\vep^*_2(z) \vep_2(z): \\
&+ :\vep^*_{\ov{2}}(z)\vep_{\ov{2}}(z):, :\vep_{\ov{1}}(w)\be^*(w): \\
&+ \om :\vep_{\ov{1}}^*(w)\vep_1^*(w): +\om^2 :\vep_{\ov{2}}^*(w)\vep_2^*(w):]\\
&= (:\vep_{\ov{1}}(w)\vep_1^*(w): + \om :\vep_{\ov{1}}^*(w)\vep_1^*(w): + :\vep_{\ov{1}}(w)\be^*(w):\\
&+ \om :\vep_{\ov{1}}^*(w)\be^*(w): + 2 \om^2:\vep_{\ov{2}}^*(w)\vep_2^*(w):) \d (z-w)\\
&= 2(:\vep_{\ov{1}}(w)\be^*(w): \om :\vep_{\ov{1}}^*(w)\vep_1^*(w): \\
&+\om^2 :\vep_{\ov{2}}^*(z)\vep_2^*(w):)\d (z-w) = 2X(\al_0,w)\d (z-w),
\end{align*}
since $:\vep_{\ov{1}}(w)\vep_1^*(w): = :\vep_{\ov{1}}(w)\be^*(w):$ and $:\vep_{\ov{1}}^*(w)\be^*(w): = :\vep_{\ov{1}}^*(w)\vep_1^*(w):$.
\begin{align*}
[\al_0(z), X(-\al_0,w)] =&[:\vep^*_1(z) \vep_1(z): + :\be^*(z) \be(z):+ :\vep^*_2(z) \vep_2(z): \\
&+ :\vep^*_{\ov{2}}(z)\vep_{\ov{2}}(z):, :\be(w)\vep_{\ov{1}}^*(w): \\
&+ \om^2 :\vep_1(w)\vep_{\ov{1}}(w): +\om :\vep_2(w)\vep_{\ov{2}}(w):]\\
&= - 2(:\be(w)\vep_{\ov{1}}^*(w): + \om^2 :\vep_1(w)\vep_{\ov{1}}(w): \\
&+\om :\vep_2(w)\vep_{\ov{2}}(z):)\d (z-w) = - 2X(-\al_0,w)\d (z-w).
\end{align*}
\begin{align*}
[\al_1(z), X(\al_2,w)] =& [:\vep_1(z)\vep_1^*(z): + :\vep_2^*(z)\vep_2(z) + 2:\vep_{\ov{2}}(z)\vep_{\ov{2}}^*(z):,\\
&:\vep_2(w)\vep_{\ov{2}}^*(w):]\\
&= (- :\vep_2(w)\vep_{\ov{2}}^*(w): + 2:\vep_{\ov{2}}^*(w)\vep_2(w):) \d (z-w)\\
&= -3 :\vep_2(w)\vep_{\ov{2}}^*(w):\d (z-w) = a_{1,2}X(\al_2,w)\d (z-w).
\end{align*}
\begin{align*}
[\al_1(z), X(-\al_2,w)] =& [:\vep_1(z)\vep_1^*(z): + :\vep_2^*(z)\vep_2(z) + 2:\vep_{\ov{2}}(z)\vep_{\ov{2}}^*(z):,\\
&:\vep_{\ov{2}}(w)\vep_2^*(w):]\\
&= (- :\vep_2^*(w)\vep_{\ov{2}}(w): + 2:\vep_{\ov{2}}(w)\vep_2^*(w):) \d (z-w)\\
&= 3 :\vep_{\ov{2}}(w)\vep_2^*(w):\d (z-w) = - a_{1,2}X(-\al_2,w)\d (z-w).
\end{align*}
\begin{align*}
[X(\al_0, z), X(-\al_0, w)] =&[:\vep_{\ov{1}}(z)\be^*(z):+ \om :\vep_{\ov{1}}^*(z)\vep_1^*(z): +\om^2 :\vep_{\ov{2}}^*(z)\vep_2^*(z):,\\
&:\be(w)\vep_{\ov{1}}^*(w):+ \om^2 :\vep_1(w)\vep_{\ov{1}}(w): +\om :\vep_2(w)\vep_{\ov{2}}(w):]\\
&= (:\be^*(w)\be(w):+:\vep_{\ov{1}}(w)\vep_{\ov{1}}^*(w): + :\vep_{\ov{1}}^*(w)\vep_{\ov{1}}(w)\\
&+:\vep_1^*(w)\vep_1(w): + :\vep_2^*(w)\vep_2(w): + \\
&:\vep_{\ov{2}}^*(w)\vep_{\ov{2}}(w):)\d (z-w)+ 3\partial_w\d (z-w)\\
&= (:\be^*(w)\be(w):+:\vep_1^*(w)\vep_1(w): + :\vep_2^*(w)\vep_2(w):\\
&+ :\vep_{\ov{2}}^*(w)\vep_{\ov{2}}(w):)\d (z-w)+ 3\partial_w\d (z-w)\\
&= \al_0(w)\d (z-w) +3\partial_w\d (z-w)\not{c}.
\end{align*}
\begin{align*}
[X(\al_2, z), X(-\al_2, w)] =& [:\vep_2(z)\vep_{\ov{2}}^*(z):, :\vep_{\ov{2}}(w)\vep_2^*(w):]\\
&= (:\vep_2(w)\vep_2^*(w): + :\vep_{\ov{2}}^*(w)\vep_{\ov{2}}^*(w):)\d (z-w)\\
&+ \partial_w \d (z-w)\\
&=\al_2(w)\d (z-w) + \partial_w\d (z-w)\not{c}.
\end{align*}
\end{proof}


\bibliographystyle{amsalpha}

\end{document}